\documentclass{article}
\usepackage[utf8]{inputenc}
\usepackage{graphicx}
\usepackage{amsmath}
\usepackage{amssymb, amsfonts, amscd, amsthm}
\usepackage{enumerate}
\usepackage[T2A]{fontenc}
\usepackage[english]{babel}
\newtheorem{theorem}{Theorem}
\newtheorem{proposition}{Proposition}
\newtheorem{lemma}{Lemma}
\newtheorem{corollary}{Corollary}

\theoremstyle{definition}
\newtheorem{def_on}{Definition}

\theoremstyle{remark}
\newtheorem{remark}{Remark}

\title{C*-algebra positive element invertibility criteria in terms of $L_1$-norms equivalence}

\author{Novikov Andrei}

\begin{document}
\Large
\maketitle

\begin{abstract}
We prove that the $L_1$-norms associated with a positive element $a$ of a unital C*-algebra are equivalent to the norm of C*-algebra if and only if $a$ is invertible.
\end{abstract}

keywords: C*-algebra, von Neumann algebra, invertibility, $L_1$-norm, noncommutative,  positive element

\section*{Introduction}

In \cite{Nov2017} we have described the construction of $L_1$-norms and $L_1$ type spaces assotiated with positive operator affiliated with von Neumann algebra. Which lead to some results for the measures on orthoideals \cite{Nov16}. As a side-effect of this work we obtained the caracterization of positive central elements by inequalities \cite{Nov15}.

In the other branch on research \cite{NovEsk2017, Esk18} we have studied different compositions of inductive and projective limits of Banach spaces of measurable functions with order unities, which essentially questioned us on what is the necessary and sufficient conditions for $L_1$-norms for operators $a^n$ and $a^{n+1}$ to be equivalent. This article gives us the answer.

\section{Definitions and Notation}
Throughout this paper we adhere the following notation: $\mathcal{A}$ denotes a C*-algebra,
$\mathcal{A}^\mathrm{sa}$ denotes self-adjoint part of $\mathcal{A}$,
$\mathcal{A}^*$ is the space of all continuous linear functionals on $\mathcal{A}$,  $\mathcal{A}_\mathrm{h}^*$ is its Hermitian part,
$\mathcal{A}^+$ denotes the positive cone of $\mathcal{A}$, $\mathcal{A}^*_+$ denotes the cone of the positive continuous linear functionals on $\mathcal{A}$. 

\begin{def_on}
For a positive element $a\in \mathcal{A}$ we consider seminorm $r_a$ and a norm $\|\cdot\|^a$ defined by the equations
$$r_a(f)=\inf\{f_1(a)+f_2(a)\ | $$
$$| \ f=f_1-f_2\ (f_1,\ f_2\in \mathcal{A}^*_+)\}\ (f\in \mathcal{A}^*_\mathrm{h}).$$
\end{def_on}

\begin{def_on}
On $\mathcal{M}_*^h$ we consider the seminorm
$$r_a(\varphi)=\inf\{\varphi_1(a)+\varphi_2(a)\ |$$
$$|\ \varphi=\varphi_1-\varphi_2\ (\varphi_1,\ \varphi_2\in \mathcal{M}_*^+)\} (\varphi\in \mathcal{M}_*^\mathrm{h}).$$
\end{def_on}

\section{Preliminaries}
For a C*-algebra $\mathcal{A}$ there always exist the universal enveloping von Neumann algebra, which will be deonted by $\mathcal{N}$~\cite[III.2]{Takesakii}.
By $\pi$ and $\pi'$ we denote the morphisms $\pi: \mathcal{A}\mapsto \mathcal{N}$
and $\pi': \mathcal{A}^* \mapsto \mathcal{N}_*$.

\begin{lemma}\label{prop3}
For $a\in\mathcal{A}^+$ the equality 
$$r_{\pi(a)}(\pi'(f))=r_a(f)\ (f\in \mathcal{A}^*_\mathrm{h})$$
holds.
\end{lemma}
\begin{proof}
Let $f\in \mathcal{A}^*_\mathrm{h}$. By definition $$r_{\pi(a)}(\pi'(f))=
\inf\{\varphi_1(\pi(a))+\varphi_2(\pi(a))\ |$$
$$|\ \pi'(f)=\varphi_1-\varphi_2\ (\varphi_1,\ \varphi_2\in \mathcal{N}_*^+)\}.$$
Since for any $\varphi\in \mathcal{N}_*$ and any $x\in \mathcal{A}$ the equality  $$\varphi(\pi(x))=((\pi')^{-1}(\varphi))(x)$$ holds, and since $\pi'$ is the isometrical isomorphism of  $\mathcal{A}^*$ onto $\mathcal{N}_*$ that preserves the order, it follows that
$$r_{\pi(a)}(\pi'(f))=\inf\{\varphi_1(\pi(a))+\varphi_2(\pi(a))\ |$$
$$|\ \pi'(f)=\varphi_1-\varphi_2\ (\varphi_1,\ \varphi_2\in \mathcal{N}_*^+)\}=$$
$$=\inf\{((\pi')^{-1}(\varphi_1))(a)+((\pi')^{-1}(\varphi_2))(a)\ |$$
$$| \ f=((\pi')^{-1}(\varphi_1))-((\pi')^{-1}(\varphi_2))\ (\varphi_1,\ \varphi_2\in \mathcal{N}_*^+)\}=$$
$$=\inf\{f_1(a)+f_2(a)\ |\ f=f_1-f_2\ (f_1,\ f_2\in \mathcal{A}^*_+)\}=r_a(f).$$
\end{proof}

\begin{theorem}[Proposition~3 \cite{SkvTik98}]
For $a\in \mathcal{M}^+$ and $\varphi\in \mathcal{M}_*^\mathrm{h}$ the equality
$$r_a(\varphi)=\|a^\frac{1}{2}\varphi a^\frac{1}{2}\|$$
holds.
\end{theorem}

\begin{corollary}
Let $a\in\mathcal{A}^+$ and $f\in \mathcal{A}^*_\mathrm{h}$, the the equality $$r_a(f)=\|a^\frac{1}{2}f a^\frac{1}{2}\|$$
holds.
\end{corollary}

\begin{remark}
If $\mathcal{A}$ is unital, then $$\|f\|=r_{\mathbf{1}}(f) \text{ for any } f\in \mathcal{A}^*_\mathrm{h}.$$
\end{remark}

\begin{theorem}[Theorem 1, \cite{Nov2017}]\label{faithfull_v_n_a}
Let $a\in \mathcal{M}^+$ then the the seminorm $r_a$ if faithfull (i.e. is a norm) on $\mathcal{M}_*^\mathrm{h}$ if and only if $\ker a=\{\overrightarrow{0}\}$, i.e. $a$ is injective. 
\end{theorem}

\begin{theorem}[Theorem 15, \cite{Nov2017}]\label{faithfull}
Let $a\in \mathcal{A}^+$ then the seminorm $r_a$ is faithfull (i.e. is a norm) on $\mathcal{A}^*_\mathrm{h}$ if and only if for any $f\in \mathcal{A}^*_+\setminus{\{\mathit{0}\}}$ the inequality $f(a)>0$ holds.
\end{theorem}

If $r_a$ is faithfull, then we denote it as $\|\cdot\|_a$.

\section{Main results}

\begin{lemma}\label{topology_inclusion}
Let $a\in \mathcal{A}^+$ and for any $\alpha, \beta \in \mathbb{R}^+\setminus\{0\}$ such that $\alpha<\beta$, then the inequality 
$r_{a^\beta}(f)\leq \|a^{\beta-\alpha}\|r_{a^\alpha}(f)$ holds for any $f\in \mathcal{A}^*_\mathrm{h}$.
\end{lemma}
\begin{proof}
Let $f\in \mathcal{A}^*_\mathrm{h}$, then
 $$r_{a^\beta}(f)=\|a^\frac{\beta}{2}f a^\frac{\beta}{2}\|=\| a^\frac{\beta-\alpha}{2} a^\frac{\alpha}{2} f a^\frac{\alpha}{2} a^\frac{\beta-\alpha}{2}\|\leq$$
 $$\leq\|a^\frac{\beta-\alpha}{2}\|^2\|a^\frac{\alpha}{2} f a^\frac{\alpha}{2}\|=\|a^{\beta-\alpha}\|r_{a^\alpha}(f).\qedhere$$
\end{proof}

\begin{lemma}\label{semiinvertible_a_v_n}
Let $a\in\mathcal{M}^+$ then the following conditions are equivalent:

\begin{enumerate}[(i)]
    \item $\exists \alpha, \beta\in \mathbb{R}^+$ ($\alpha\neq\beta)$ such that $r_{a^\alpha}$ and $r_{a^\beta}$ are equivalent to each other;
    \item $\forall \alpha \in \mathbb{R}^+$  $r_{a^\alpha}$ is equivalent to $r_{\mathrm{rp}(a)}$, where $\mathrm{rp}(a)$ is range projection of operator $a$
\end{enumerate}
\end{lemma}

\begin{proof}
 $(i)\implies (ii)$. Without loss of generality we consider that $\alpha < \beta$. 
 Since $$f(a^\alpha)=r_{a^\alpha}(f) \leq C r_{a^\beta}(f)=C f(a^\beta) \text{ for all } f\in \mathcal{A}^*_+$$ it follows that $a^\alpha \leq C a^\beta$. Therefore, for any natural $n$
 $$\left(\frac{1}{n}\mathbf{1}+a^\frac{\alpha}{2}\right)^{-1} a^\alpha \left(\frac{1}{n}\mathbf{1}+a^\frac{\alpha}{2}\right)^{-1}
 \leq $$
 $$\leq C \left(\frac{1}{n}\mathbf{1}+a^\frac{\alpha}{2}\right)^{-1} a^\beta \left(\frac{1}{n}\mathbf{1}+a^\frac{\alpha}{2}\right)^{-1}.$$
 Let $p=\mathrm{rp}(a)=s$-$\lim\limits_n a(\frac{1}{n}\mathbf{1}+a)^{-1}$, then $p\leq C p a^{\beta-\alpha} p$, thus
 $$p=p^\xi\leq C^\xi (p a^{\beta-\alpha} p)^\xi \text{ for each } \xi> 0.$$ Let $\gamma> 0$ and $\xi=\frac{\gamma}{\beta-\alpha}$, 
 then $C^{\frac{\gamma}{\alpha-\beta}}p\leq p a^\gamma p$.
 Therefore, $$C^{\frac{\gamma}{\alpha-\beta}}r_p\leq r_{pa^\gamma p}=r_{a^\gamma}\leq \|a^\gamma\|r_p.$$
 
 The implication $(ii)\implies (i)$ is evident.
\end{proof}

\begin{proposition}\label{two_conditions}
For $a\in \mathcal{A}^+$ the following conditions areequivalent: 
\begin{enumerate}[(i)]
    \item $\exists \alpha, \beta\in \mathbb{R}^+$ ($\alpha\neq\beta)$ such that $r_{a^\alpha}$ and $r_{a^\beta}$ are equivalent to each other;
    \item $\forall \alpha \in \mathbb{R}^+$  $r_{a^\alpha}$ is equivalent to $r_a$.
\end{enumerate}
\end{proposition}
\begin{proof}
 $(i)\implies (ii).$ By Lemma~\ref{prop3} the equivalence of 
$r_{a^\alpha}$ and $r_{a^\beta}$ implies the equivalence of
$r_{\pi(a)^\alpha}$ и $r_{\pi(a)^\beta}$. From Lemma~\ref{semiinvertible_a_v_n} it follows that $r_{\pi(a)^\gamma}$ is equivalent to $r_{\mathrm{rp}\pi(a)}$ for any $\gamma\in \mathbb{R}^+$,
including $\gamma=1$. Thus, $r_{\pi(a)^\alpha}$ is equivalent to $r_{\pi(a)}$ for each $\alpha\in \mathbb{R}^+$.
Again, by Lemma~\ref{prop3}, the latter implies that $r_{a^\alpha}$ is equivalent to $r_{a}$ for each $\alpha \in \mathbb{R}^+$.

 The implication $(ii)\implies (i)$ is evident.
\end{proof}

\begin{theorem}\label{invertible_topology}
Let $\mathcal{A}$ be unital C*-algebra $\mathcal{A}$ and $a\in \mathcal{A}^+$, then the following conditions are equivalent:

\begin{enumerate}[(i)]
    \item $a$ is invertible
    \item $\forall \alpha \in \mathbb{R}$ $\|\cdot\|_{a^\alpha}$ is equivalent to $\|\cdot\|$;
    \item $\forall \alpha \in \mathbb{R}^+$  $\|\cdot\|_{a^\alpha}$ is equivalent to $\|\cdot\|$;
    \item $\forall \varphi\in \mathcal{A}^*_+\setminus\{\mathit{0}\}\ (\varphi(a)>0)$ and there exist such  $\alpha$, $\beta\in \mathbb{R}^+ (\alpha\neq \beta)$ that $\|\cdot\|_{a^\alpha}$ is equivalent to $\|\cdot\|_{a^\beta}$.
\end{enumerate}
\end{theorem}
\begin{proof}

$(i)\implies (ii).$ Let $\alpha\in \mathbb{R}$. If $a$ is invertible, then there exists $a^{-1}\in \mathcal{A}^+$ and, on one hand, 
$$\|f\|_{a^\alpha}=\|a^\frac{\alpha}{2}f a^\frac{\alpha}{2}\|\leq \|a^\alpha\|\|f\|\ (\text{for all } f\in \mathcal{A}^*_\mathrm{h}),$$
and on the other hand, $$\|f\|=\|a^\frac{-\alpha}{2}a^\frac{\alpha}{2}f a^\frac{\alpha}{2} a^\frac{-\alpha}{2}\|\leq \|a^{-\alpha}\|\|f\|_{a^\alpha}\ (\text{for all } f\in \mathcal{A}^*_\mathrm{h}).$$

The implications $(ii) \implies (iii)$, $(iii) \implies (iv)$ are evident.

$(iv) \implies (i)$. Since $\|\cdot\|_{a^\alpha}$ is equivalent to $\|\cdot\|_{a^\beta}$,
it follows that there exist $c$ and $C$ such that $$c\|f\|_{a^\alpha}\leq \|f\|_{a^\beta}\leq C\|f\|_{a^\alpha}\ (\text{for all } f\in \mathcal{A}^*_\mathrm{h} ).$$ Without loss of generality we assume, that $\alpha <\beta$. For arbitrary  $f\in \mathcal{A}^*_+$ 
the inequalities $$f(c a^\alpha)\leq f(a^\beta) \leq f(C a^\alpha)$$ holds, thus $$c a^\alpha \leq a^\beta \leq C a^\alpha,$$
hence for all natural $n\in \mathbb{N}$ $$c \left(\frac{1}{n}+a^\alpha\right)^{-1} a^\alpha \leq \left(\frac{1}{n}+a^\alpha\right)^{-1} a^\beta\leq C \left(\frac{1}{n}+a^\alpha\right)^{-1} a^\alpha.$$
We pass to the limit by $n$ and get $$c\mathbf{1}\leq a^{\beta-\alpha} \leq C\mathbf{1},$$ thus $a^{\beta-\alpha}$ is invertible,
as well as any $a^\gamma$ ($\gamma \in \mathbb{R}$).
\end{proof}

\begin{corollary}
Let $\mathcal{A}$ be unital C*-algebra $\mathcal{A}$ and $a\in \mathcal{A}^+$, then the following conditions are equivalent:

\begin{enumerate}[(i)]
    \item $a$ is invertible;
    \item $\|\cdot\|_a$ is equivalent to $\|\cdot\|;$
    \item $\forall \alpha\in \mathbb{R}^+$ $\|\cdot\|_{a^\alpha}$ is equivalent to $\|\cdot\|.$
\end{enumerate}
\end{corollary}

\section*{Acknowledgment}
Research is partially supported by Russian Foundation for Basic Research grant
18-31-00218 (мол\_а).

\end{document}